\documentclass[oneside,10pt, a5paper]{article}
\usepackage[cp1251]{inputenc}
\usepackage[english]{babel}
\usepackage{amsmath, amssymb, accents,amsthm}%
\usepackage{graphicx}
\usepackage[top=2cm, left=1.5cm]{geometry}
\usepackage{cite}

\DeclareMathOperator{\rank}{rank}

\makeatletter
\newcommand{\subalign}[1]{%
  \vcenter{%
    \Let@ \restore@math@cr \default@tag
    \baselineskip\fontdimen10 \scriptfont\tw@
    \advance\baselineskip\fontdimen12 \scriptfont\tw@
    \lineskip\thr@@\fontdimen8 \scriptfont\thr@@
    \lineskiplimit\lineskip
    \ialign{\hfil$\m@th\scriptstyle##$&$\m@th\scriptstyle{}##$\hfil\crcr
      #1\crcr
    }%
  }%
}
\makeatother

    \advance\textwidth  by -31truemm     
    \advance\textheight by -29truemm    

\newcommand{\bs}{\boldsymbol}
\newcommand{\bal}{\bs\alpha}
\newcommand{\bbt}{\bs\beta}
\newcommand{\bg}{\bs\gamma}
\newcommand{\be}{\bs e}

\newcommand{\bv}{\bs v}

\newcommand{\x}{\times}
\newcommand{\br}{\bs r}

\newcommand{\om}{\omega}

\newcommand{\bom}{\bs\om}

\renewcommand{\d}{\partial}
\newcommand{\I}{{\bf I}}
\newcommand{\J}{{\bf J}}

\newcommand{\E}{{\bf E}}

\newcommand{\Q}{{\bf Q}}
\newcommand{\B}{{\bf B}}
\newcommand{\q}{\quad}
\newcommand{\qq}{\qquad}
\newcommand{\g}{{\rm g}}
\newcommand{\lm}{\lambda}
\newcommand{\gam}{\gamma}
\newcommand{\vfi}{\varphi}

\DeclareMathOperator{\diag}{diag}
\DeclareMathOperator{\const}{const}
\newcommand{\eps}{\varepsilon}

\newcommand{\vth}{\vartheta}
\renewcommand{\theta}{\vartheta}

\renewcommand{\j}{\mathrm i}

\DeclareMathSymbol{\widetildesym}{\mathord}{largesymbols}{"65}

\newtheorem{thm}{Теорема}
\newtheorem*{pro}{Proposition}

\newtheorem{utv}{Statement}
\newtheorem*{hyp}{Hypothesis}
\theoremstyle{remark}
\newtheorem*{rem}{Remark}

\makeatletter
\newcommand{\settheoremtag}[1]{
  \let\oldthetheorem\thethm
  \renewcommand{\thethm}{#1}
  \g@addto@macro\endthm{
    \addtocounter{thm}{-1}
    \global\let\thethm\oldthetheorem}
  }
\makeatother

\title{Bifurcation analysis of the problem of~a~``rubber'' ellipsoid of revolution
rolling~on~a~plane}
\author{A.\,A.Kilin, E.\,N.Pivovarova\\
{\small Ural Mathematical Center,}\\
{\small Udmurt State University,}\\
{\small ul. Universitetskaya 1, Izhevsk, 426034 Russia}}
\date{}

\begin{document}

\maketitle
\begin{abstract}
This paper is concerned with the problem of~an ellipsoid of revolution rolling
on a horizontal plane under the assumption that there is no slipping at the point of contact
and no spinning about the vertical. A~reduction of the equations of motion
to a fixed level set of first integrals is performed. Permanent rotations
corresponding to the rolling of an ellipsoid in a circle or in a straight line are found.
A linear stability analysis of~permanent rotations is carried out.
A complete classification of possible trajectories of the reduced system is performed using
a bifurcation analysis. A classification of the trajectories of the center of mass of the
ellipsoid depending on parameter values and initial conditions is performed.

\end{abstract}

\setcounter{tocdepth}{2}

\tableofcontents

\section{Introduction\label{sec1}}

Analysis of the motion of rigid bodies on a plane is one of the classical problems in mechanics.
To describe the~motion of a rigid body in the first approximation, use is often made of
conservative models: the model of an absolutely smooth plane (the point of contact can slip)
and the model of an absolutely rough plane (the classical nonholonomic model in which there
is no slipping at the point of contact). Another conservative
model of contact is that of a rolling ``rubber'' body --- it
assumes that there is no slipping at the point of
contact and no spinning about the vertical. Early studies of
the rolling motion of bodies without spinning go back to the classical works on modern mechanics:
this model was proposed by Hadamard~\cite{hadamard1895mouvements} and
developed by Beghin~\cite{beghin1929conditions}. In~\cite{ehlers2008rubber},
the term ``rubber rolling'' was introduced to describe this
motion model.
A systematization of results on the rolling motion of rigid
bodies using the above-mentioned model was presented in~\cite{borisov2008conservation,borisov2013hierarchy},
where existing tensor invariants of the system depending on
the shape of the body and the mass distribution in it were found.

The main feature of the problem of a body of revolution
moving on a plane is its integrability regardless of the
shape of the body. The same holds true for an~absolutely smooth
plane~\cite{appell1900lintegration,appell1904traite},
for an absolutely rough plane~\cite{chaplygin2002motion},
and for the rubber body rolling model~\cite{borisov2008conservation}.
Moreover, if the system is subject to one nonholonomic constraint which prohibits sliding in the direction of the projection of the axis of rotation, then such a problem is also integrable. This was shown in~\cite{kilin2023integrable}, where the authors addressed the problem of the rolling motion of a sphere with axisymmetric mass distribution. It was shown that this system admits a redundant set of first integrals and an invariant measure, and that one of the found integrals is similar to the additional integral found in~\cite{kozlov1978theorems} in the problem of a circular disk rolling on smooth ice. This allowed a reduction to a system with one degree of freedom and made it possible to show that all nonsingular trajectories are periodic functions of time.
It is interesting that, if the system is subject to a constraint prohibiting sliding in the direction perpendicular to the axis of rotation, then the problem becomes nonintegrable~\cite{kilin2023problem}.                                                                                                              For the case of the sphere with axisymmetric mass distribution it is shown that, in the general case, this system is nonintegrable (one additional integral of motion and an invariant measure are needed for the system to be integrable). It is numerically shown that the system exhibits chaotic dynamics. In addition, the problem of a body of revolution rolling on a plane
 remains integrable after adding an additional potential that is invariant under rotation about the vertical~\cite{borisov2022top,kilin2023stability,hier2002}.

 There is a related problem, that of a homogeneous sphere rolling without
slipping on a surface of revolution, which similar to the problem of a body of revolution rolling on a plane.
This problem is also integrable regardless of the form of the surface. This is shown in~\cite{routh,borisov2002rolling}
using the same methods of analysis (in particular, reduction by symmetries).
In~\cite{dalla2022dynamics,fasso2022some} the rolling motion of a sphere on an arbitrary
rotating surface of revolution is investigated, and the authors of~\cite{borisov2019nonholonomic,ivanova2020non}
examine integrable cases of this problem: the rolling motions of a sphere on the rotating
surface of a cone and a circular cylinder. In~\cite{dragovic2022spherical,dragovic2023spherical}
a new interesting problem concerning the rolling of spheres between two spherical surfaces
is addressed and cases of its integrability are found, and in~\cite{borisov2012generalized}
the integrability of the related problem of the motion of a rigid body in a spherical
suspension is shown~\cite{fedorov1988motion}.

Compared to the study of body motion using the~classical nonholonomic model prohibiting
only slipping at the point of contact, much less attention has been given in the literature to
the dynamics of bodies using the~rubber body model. In particular, analysis of
rigid body dynamics using this model was made in~\cite{cendra2006rolling,cendra2010impulsive,koiller2007rubber,ehlers2008rubber}.
In~\cite{bizyaev2018invariant}, the rubber body model was employed to examine the dynamics of an
unbalanced disk. The rolling motion of spheres with different mass distributions was studied
in~\cite{cendra2006rolling,cendra2010impulsive,koiller2007rubber,ehlers2008rubber,bolsinov2012rolling,borisov2013rolling,kazakov2013strange,borisov2016regular,mamaev2020dynamics}, and the motion of a triaxial ellipsoid was examined in~\cite{bizyaev2013integrability}.

However, the ``rubber'' motion model is applied not only
in handling some model problems concerning, e.g., disks or spheres, but also in
simulating the motion of~mobile robots. In particular, the problem of using various
motion models for simulating the dynamics of mobile robots is discussed
in~\cite{ylikorpi2014gyroscopic}, where it is noted that the~rolling motion of a sphere
without slipping or spinning
is the classical ball-plate problem~\cite{jurdjevic1993geometry,bicchi1995planning}.
In \cite{ylikorpi2014gyroscopic}, the authors also note that in many studies
the vertical component of the angular velocity of the spherical robot is assumed
to be zero since the possibility of omnidirectional motion allows motion control without using
vertical rotations. Such a motion model is used in analyzing the dynamics of a spherical
robot with a pendulum drive~\cite{zheng2011control,cai2012path}, a spherical robot with an
internal platform~\cite{alves2003design,zhan2011design}, a spherical robot with internal omniwheels~\cite{chen2012design,karavaev2022spherical}, and a spherical robot actuated by changing the gyrostatic momentum~\cite{artemova2020dynamics}. In addition, some authors~\cite{zhan2011design,chen2012design,marigo2000rolling,mukherjee2002motion},
who use the ``rubber'' motion model consider only kinematics without taking into account the
system dynamics.

In this paper, we carry out a qualitative analysis of~the motion of an ellipsoid of revolution
on a horizontal plane using the model of rubber body. Although we
restrict our attention to a special case of~body shape, many conclusions of this paper
will be valid for a body of an arbitrary shape as well. In addition, the analysis of the
dynamics in specific problems remains of interest to researchers (see, e.g., the recent
paper on qualitative analysis of the rolling motion of~a~torus
on a plane~\cite{garcia2024analytical}).

The paper is organized as follows. Section~\ref{sec2} presents equations of motion,
invariants, and a reduction to a~system with one degree of freedom for a body of revolution
of arbitrary form. In Section~\ref{sec3},
an explicit form of the equations of motion for an ellipsoid of revolution is presented.
Section~\ref{sec4} provides a qualitative analysis of the dynamics of the ellipsoid and
a complete classification of motions depending on parameters and values of first integrals.
In Section~\ref{sec5}, an analysis of the trajectories of~the center of mass of the ellipsoid
in absolute space is made and conclusions concerning the boundedness and unboundedness
of the trajectories are drawn.

\section{Equations of motion for a body of~revolution on a plane\label{sec_equ}}

\subsection{Configuration space and kinematic relations\label{sec2}}

Consider the rolling motion of a dynamically symmetric rigid body of revolution
of mass $m$ on an
absolutely rough
horizontal plane
(Fig.~\ref{fig1}). We will assume that the body rolls without slipping (the velocity
of the point of contact is zero) and without spinning (the projection of the angular velocity
of the body onto the vertical is zero), and has only one point of contact, $P$, with the
supporting plane.

\begin{figure}[!ht]
\centering\includegraphics{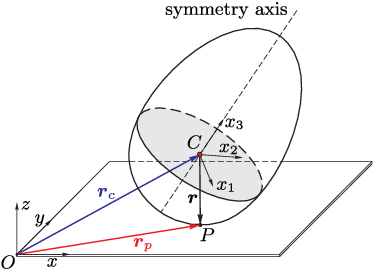}
\caption{A body of revolution on a plane.\label{fig1}}
\end{figure}

We introduce two coordinate systems:
\begin{enumerate}\itemsep=-2pt
\item[--] a fixed (inertial) coordinate system $Oxyz$ with unit vectors $\bal, \bbt,
\bg$, where $\bg$ is a vertical unit vector;
\item[--] a moving coordinate system $Cx_1x_2x_3$ with vectors
    $\be_1,\be_2,\be_3$ rigidly attached to the body, and with origin at its center of mass,
    $C$, such that the axis $Cx_3$ is directed along the symmetry axis.
\end{enumerate}

We will specify the position of the body of revolution by the coordinates of its center
of mass $\br_c=(x_c, y_c, z_c)\in \mathbb{R}^3$
in the coordinate system $Oxyz$, and its orientation in space, by the orthogonal matrix $\Q$,
whose columns are the projections of the unit vectors $\bal,\,\bbt,\,\bg$ onto the
axes of the moving coordinate system $Cx_1x_2x_3$
$$
{\bf Q}=\left(
\begin{aligned}
&\alpha_1 & \beta_1 & &\gamma_1 \\
&\alpha_2 & \beta_2 & &\gamma_2 \\
&\alpha_3 & \beta_3 & &\gamma_3
\end{aligned}
\right)
\in SO(3).
$$
 Hence, the configuration space of the problem of the free motion of the body of revolution
 is the product
$\mathcal{N}=\{(\br_c,\Q)\}=\mathbb{R}^3\x SO(3)$.

To describe the dynamics of the body on a plane, we introduce the vector of the angular
velocity $\bom=(\om_1,\om_2,\om_3)$ and the vector of the velocity of the center of
mass $\bv=(v_1,v_2,v_3)$ and refer them to the moving coordinate system $Cx_1x_2x_3$\footnote{Here and below, unless
otherwise specified, all vectors will be referred to the moving coordinate system
$Cx_1x_2x_3$.}. Using these
velocities (quasi-velocities), the evolution of the vector $\br_c$ and the matrix $\Q$ can be
written as  the kinematic relations
\begin{equation}\begin{gathered}
\dot\bal=\bal\x\bom,\q\dot\bbt=\bbt\x\bom,\q\dot\bg=\bg\x\bom,\\
\dot\br_c=\Q^\top\bv.
\end{gathered}\label{kin}
\end{equation}

\subsection{Constraint equations and dynamical equations}
We assume that the body rolls without loss of contact with the plane. This assumption
implies that the following holonomic constraint is imposed on the system:
\begin{equation}z_c+(\br,\bg)=0,\label{constr_hol}\end{equation}
  where $\br$ is the radius vector of the point of contact (see Fig.~\ref{fig1}), which can
  be expressed in terms of the components of the vector $\bg$. Here and in what follows, by $(\bs\cdot,\bs\cdot)$ we denote the scalar product. In the case of an arbitrary
  axisymmetric surface of the body, the dependence $\br(\bg)$ can be represented as
\[\br(\bg)=(\chi_1(\gam_3)\gam_1,\chi_1(\gam_3)\gam_2,\chi_2(\gam_3)),\]
where $\chi_1(\gam_3)$ and $\chi_2(\gam_3)$ are arbitrary functions related to each other by~\cite{hier2002}
\[\frac{d \chi_2}{d \gam_3}=\chi_1-\frac{1-\gam_3^2}{\gam_3}\frac{d \chi_1}{d \gam_3}.\]
This relation is obtained from the condition $(\dot\br,\bg)=0$.

The no slip condition at the point of contact is described by a
constraint of the
form
\begin{equation}\bs f=\bv+\bom\x\br={\bs 0}.\label{constr}\end{equation}

\begin{rem}
Two projections of equation~\eqref{constr}, $(\bs f, \bal)$ and $(\bs f, \bbt)$, are
nonholonomic constraints, and the third, $(\bs f, \bg)$, is the time derivative of
equation~\eqref{constr_hol}.
\end{rem}

The condition that there be no spinning about the vertical is described by
another nonholonomic constraint:
\begin{equation}g = (\bom,\bg)=0.\label{constr2}\end{equation}

The kinetic energy $T$ and the potential energy $U$ of the system have the form
\[T = \frac12 (\bom, \I \bom) + \frac12 m(\bv, \bv),\q U=-m\g (\br,\bg),\]
where $\I=\diag(i_1,i_1,i_3)$ is the principal central tensor of inertia of the body and
$\g$ is the free-fall acceleration.

\begin{rem}
Note that, due to the axial symmetry, the potential energy $U$ of the system
depends only on the inclination of the symmetry axis relative to the vertical, that is,
the component $\gam_3$.
\end{rem}

We write the equations of motion of the system in the absence of external forces as Lagrange
equations in quasi-velocities with constraints~\cite{poinc,dtt}
\begin{equation}
\begin{aligned}
&\dfrac{\rm d}{{\rm d} t}\left(\dfrac{\d \mathcal{L}}{\d \bv}\right) + \bom\x\dfrac{\d \mathcal{L}}{\d \bv}
=\left(\dfrac{\d  \bs f}{\d \bv}\right)^{\!\!\top}\bs\lm ,&\\
&\dfrac{\rm d}{{\rm d} t}\!\left(\dfrac{\d \mathcal{L}}{\d \bom}\right)\! +
\bom\x\dfrac{\d \mathcal{L}}{\d \bom} + \bv\x\dfrac{\d \mathcal{L}}{\d
\bv} + \bg\x\dfrac{\d \mathcal{L}}{\d
\bg}  =\left(\dfrac{\d  \bs f}{\d \bom}\right)^{\!\!\top}\bs\lm + \left(\dfrac{\d  g}{\d \bom}\right)\mu,&
\end{aligned}\label{dallagrgen}
\end{equation}
where $\mathcal{L}=T-U$ is the Lagrangian function of the system, and
$\bs\lm$, $\mu$ are the undetermined multipliers corresponding to the reactions of the
constraints~\eqref{constr}
and~\eqref{constr2}, respectively. The undetermined multipliers can be found from the common
solution of equations~\eqref{dallagrgen} and the time derivatives of the
constraints~\eqref{constr} and \eqref{constr2}:
\begin{equation*}
\begin{gathered}
\bs\lm = -m((\bom\x\br)^{\bs\cdot} + \bom\x(\bom\x\br)),\\
\mu = \frac{(\bom\x\J\bom + m\br\x(\bom\x\dot\br) + m\g\bg\x\br, \J^{-1}\bg)}{(\bg,\J^{-1}\bg)},
\end{gathered}
\end{equation*}
where $\J=\I+m\br^2\E_3 - m\br\otimes\br$ is the tensor of inertia of the body relative to
the point of contact. Here $\E_3$ is the $3\x3$ identity matrix and the symbol
$\otimes$ stands for a tensor multiplication of the
vectors which associates to the pair of vectors
$\bs a,\,\bs
b$ the matrix ${\bf A}$ with components $A_{i, j}=a_ib_j$.

After finding the undetermined multipliers and eliminating the velocity $\bv$ (using the
constraint~\eqref{constr}), the
equations of motion~\eqref{dallagrgen} reduce to one vector equation
\begin{equation}
\J\dot\bom+\bom\x\J\bom + m\br\x(\bom\x\dot\br) +m\g\bg\x\br - \mu\bg=0.
\label{tmp2}
\end{equation}

Adding the kinematic equation of evolution of the vector $\bg$ to equation~\eqref{tmp2}, we
obtain the following closed reduced system of equations:
\begin{equation}\begin{aligned}
&\J\dot\bom+\bom\!\x\!\J\bom + m\br\!\x\!(\bom\!\x\!\dot\br) +m\g\bg\!\x\!\br - \mu\bg=0,\\
&\dot\bg=\bg\x\bom.
\end{aligned}\label{eqred}\end{equation}
By ``construction'', this system admits a first integral of motion defined by the
constraint~\eqref{constr2} $(\bom,\bg)=\const$. The flow defined by equations~\eqref{eqred}
and restricted to the zero level set of this integral describes the evolution of the variables
$\bom$ and $\bg$ in the problem of a body of revolution rolling without slipping or spinning
on a horizontal plane.

To reconstruct the velocity of the center of mass, use should be made of the constraint
equation~\eqref{constr}. The coordinate $z_c$ is defined from equation~\eqref{constr_hol},
and the remaining variables, from the quadratures~\eqref{kin}.

\goodbreak

\subsection{Invariants and symmetries\label{sec_inv}}

The equations of motion~\eqref{eqred} admit the following integrals of motion:
\begin{enumerate}
\item[--] the energy integral
 \[\mathcal{E}=\frac12(\J\bom,\bom)-m\g(\br,\bg) ;\]
 \item[--] the geometric integral
\[\mathcal{F}_0=(\bg,\bg)=1;\]
\item[--] the no-spin constraint~\eqref{constr2}
\[\mathcal{F}_1=(\bom,\bg)=0;\]
\item[--] the additional integral linear in the angular velocity
\[\mathcal{F}_2=J(\gam_3)\om_3,\,\, J(\gam_3) = \sqrt{i_1\gam_3^2+i_3(1-\gam_3^2) +m(\br,\bg)^2}.\]
\end{enumerate}
\begin{rem}
The system~\eqref{eqred} admits the energy integral $\mathcal{E}$ and the additional integral $\mathcal{F}_2$ only on the zero level set of the integral $\mathcal{F}_1$.
\end{rem}
\begin{rem}
We note that in the presence of the constraint~\eqref{constr2} the additional integral is
expressed in terms of elementary functions for an arbitrary body of revolution, whereas
in the classical nonholonomic case the additional integral is written in elementary
functions only for the spherical surface of the body.
\end{rem}

In addition, equations~\eqref{eqred} also admit the invariant measure $\rho d\bom d\bg$
with density~\cite{hier2002}
\begin{equation*}\rho(\gam_3)=(i_1+m\br^2)J(\gam_3)\label{meas}\end{equation*}
and possess the symmetry field
\begin{equation}\bs\varsigma = \om_1\frac{\d }{\d \om_2}-\om_2\frac{\d}{\d \om_1}+\gam_1\frac{\d }{\d \gam_2}-\gam_2\frac{\d}{\d \gam_1},\label{vecfield}\end{equation}
which corresponds to invariance of the system under rotations about the axis of
dynamical symmetry of the body.

{\it Thus, the problem of a body of revolution rolling on a plane without slipping or spinning
about the vertical is integrable by the Euler\,--\,Jacobi theorem and reduces to quadratures~\cite{borisov2008conservation}.}

\subsection{Reduction to a system with one degree of freedom\label{oned}}

The existence of four integrals of motion allows us to reduce the system~\eqref{eqred} to a
system with one degree of freedom.

To perform a reduction, we parameterize the vector $\bg$ by the nutation angle $\vth$ and
the angle of proper rotation $\vfi$ in the form
\begin{equation}\gam_1 = \sin\vth\sin\vfi,\q \gam_2=\sin\vth\cos\vfi,\q \gam_3=\cos\vth.\label{newvars}\end{equation}

It is easy to show that on the fixed level set of the first integrals
\[\mathcal{F}_0 = 1,\q \mathcal{F}_1 = 0, \q \mathcal{F}_2=k \] the system of
equations~\eqref{eqred} reduces to a system with one degree of freedom
\begin{equation}\begin{aligned}
&\dot\vth = p_\vth,\qq \dot{p}_\vth = \frac{G(\vth)}{B(\vth)},\\ &G(\vth) = \left(\frac{k^2\cot\vth}{\sin^2\!\vth}-\frac{B'(\vth)}{2}p_\vth^2-\frac{\partial U(\vth)}{\partial \vth}\right),
\end{aligned}\label{eqsys}\end{equation}
where $B(\vth)=i_1+m\br^2$, $U(\vth)$ is the potential energy of the system. The system~\eqref{eqsys}
admits the energy integral
\[\mathcal{E} = \frac{B(\vth)p_\vth^2}{2} +\frac{k^2}{2\sin^2\!\vth} + U(\vth).\]

In this case, the quadrature for angle $\vfi$ has the form
\begin{equation*}\dot\vfi = \frac{\om_3}{\sin^2\!\vth},\label{dfi}\end{equation*}
where
\[\om_3=\frac{k}{J(\gam_3)}=\frac{k}{J(\vth)}.\]

The values of the variables $\om_1$ and $\om_2$ can be reconstructed from the relations
\begin{equation}
\begin{aligned}&\om_1 = \dot\vth\cos\vfi - \om_3\cot\vth\sin\vfi,\\
&\om_2 = -\dot\vth\sin\vfi - \om_3\cot\vth\cos\vfi.
\end{aligned}\label{w12}
\end{equation}

Next, we consider the case of an ellipsoid of revolution and perform a complete
classification of possible motions of
this system.



\section{Explicit form of the equations of motion\\ for~an~ellipsoid of~revolution\label{sec3}}

We now turn our attention to a specific body of revolution,
an ellipsoid of revolution with semiaxes $b_1$ and $b_3$ and
with a center of mass displaced along the symmetry axis
by distance $a$ (see Fig.~\ref{fig2}).

\begin{figure*}[!ht]
\centering\includegraphics[width=\textwidth]{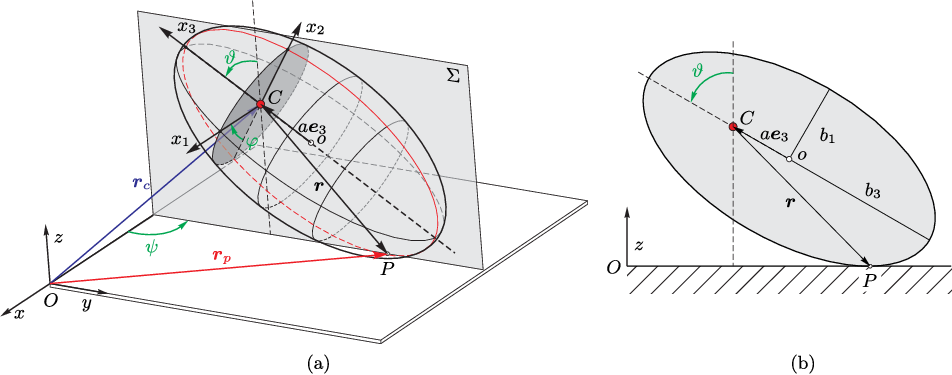}
\caption{ (a) Diagrammatic representation of an
ellipsoid on a plane. The plane
$\rm\Sigma$ is a vertical plane passing through the symmetry
axis of the ellipsoid. The red
elliptic curve is the intersection of the ellipsoid and the
plane $\rm\Sigma$. (b) Representation of the ellipsoid on the
plane in a meridional section.\label{fig2}}
\end{figure*}

The radius vector of the point of contact of the ellipsoid
can be represented as
\begin{equation*}\br=-\frac{\B\bg}{\sqrt{(\bg,\B\bg)}}-a\be_3,\qq \B=\diag(b_1^2,b_1^2,b_3^2).\label{r}\end{equation*}

\begin{rem}
In this case, the explicit expressions for the functions
$\chi_1(\gam_3)$ and $\chi_2(\gam_3)$ for the case of an
ellipsoid of revolution have the form
\begin{gather*}\chi_1(\gam_3)=-\frac{b_1^2}{\sqrt{b_1^2(1-\gam_3^2)+b_3^2\gam_3^2}},\\ \chi_2(\gam_3)=-\frac{b_3^2\gam_3}{\sqrt{b_1^2(1-\gam_3^2)+b_3^2\gam_3^2}}-a.\end{gather*}
\end{rem}

The equations of motion~\eqref{eqred} written in the new
variables depend on many parameters.
In order to decrease their number, we introduce
dimensionless parameters: by choosing the length of semiaxis
$b_3$ as a unit of length, the mass of the ellipsoid $m$ as a
unit of mass, and the quantity $\sqrt{b_3/\g}$ as a unit of
time, we obtain four dimensionless parameters
\begin{gather*}\alpha=\frac{a}{b_3}\in[0,1],\q \beta = \frac{b_1}{b_3}\in(0,\infty),\\ \nu=\frac{i_3}{i_1}\in(0,2],\q \eta=\frac{mb_3^2}{i_1}\in(0,\infty)\end{gather*}
and two dimensionless constants of the integrals
\[\eps = \frac{\mathcal{E}}{m\g b_3},\q \kappa^2 = \frac{k^2}{m\g b_3}.\]
The values $\beta\in(0,1)$ correspond to
an ``prolate'' ellipsoid, and $\beta\in(1, \infty)$, to an ``oblate'' one.
In the limiting case $\beta=0$ we obtain a thin rod,
and for $\beta\rightarrow\infty$ we obtain a thin disk. However, these cases correspond to the motion of a body with a sharp
edge, which we do not consider in this paper.

After transformation to the new variables and
nondimensionalization for the case of an ellipsoid of
revolution, the potential energy $U(\vth)$ takes the form
\begin{gather*}U(\vth)=\alpha\cos\vth + Z(\vth), \\ Z(\vth)=\sqrt{\beta^2\sin^2\!\vth+\cos^2\!\vth}>0\q\forall\,\vth\in[0,\pi],\end{gather*}
and the equations of motion~\eqref{eqsys} can be written as
\begin{equation*}\begin{aligned}
&\dot\vth = p_\vth,\\
&\dot{p}_\vth = \frac{1}{B(\vth)}\left(\frac{\kappa^2\cot\vth}{\sin^2\!\vth}-\frac{B'(\vth)}{2}p_\vth^2+\alpha\sin\vth
+\frac{(1-\beta^2)\sin\vth\cos\vth}{Z(\vth)}\right),
\end{aligned}\label{eqsysell}\end{equation*}
where
\[B(\vth) = \frac1\eta + \frac1{Z(\vth)^2}\left(\beta^4\sin^2\!\vth + (\alpha Z(\vth) - \cos\vth)^2\right).\]
We note here that the physical meaning of $Z(\vth)$ is the
height of the geometric center of the ellipsoid above the
supporting plane.

Next, we perform a complete classification of possible motions
of the resulting system.

\section{Bifurcation analysis\label{sec4}}

\subsection{General definitions}

To classify the motions of the system, we use the
topological approach described in~\cite{bolsinov2010topology,bizyaev2015topology}.
For this we consider the system~\eqref{eqred} in the
redundant coordinates $\bs z = (\bom,\bg)$, which defines the
phase flow in the space $\mathcal{M}_6=\{\bs z\}$.
This system possesses four integrals of motion $\mathcal{F}_0$, $\mathcal{F}_1$, $\mathcal{F}_2$, and $\mathcal{E}$. We restrict the system~\eqref{eqred} to the common level set of the
integrals $\mathcal{F}_0$ and $\mathcal{F}_1$
\[\mathcal{M}_4=\{\bs z\,|\, \mathcal{F}_0=1,\,\mathcal{F}_1=0\}\subset \mathcal{M}_6.\]
Let us define the integral map of this system ${\rm\Phi}\,:\,\mathcal{M}_4\rightarrow \mathbb{R}^2$
\[\bs z\rightarrow {\rm\Phi}(\bs z) = (\mathcal{F}_2(\bs z), \mathcal{E}(\bs z))\in \mathbb{R}^2,\q \bs z\in \mathcal{M}_4.\]

\begin{rem}
  As local coordinates on $\mathcal{M}_4$ we can choose,
  for example, the quantities $(\vth,\,\vfi,\,\om_1,\,\om_2)$.
  However, some solutions of interest lie on the singularities
  of the chosen coordinates ($\vth=0,\pi$). Therefore,
  it is more convenient to perform calculations in the
  redundant coordinates.
\end{rem}

The complete image of the phase space ${\rm\Phi}(\mathcal{M}_4)$
defines the \textit{region of possible motions $($RPM$)$}
of the system on the plane of first integrals $(\kappa,\eps)$. Each point on the plane of first
integrals $(\kappa,\eps)$
corresponds to the integral manifold
\[\mathcal{M}_{\kappa,\eps} \!=\! \{\bs z\,|\,\mathcal{F}_0=1,\mathcal{F}_1=0,\mathcal{F}_2(\bs z)=\kappa, \mathcal{E}(\bs z)=\eps\}\!\subset\! \mathcal{M}_6,\]
which can contain several connectedness components (generally
two-dimensional tori).

Changes in the values of the integrals $\kappa$ and $\eps$
lead to bifurcations of the integral manifolds. These
bifurcations correspond to a \textit{critical set} of the
integral map which  is defined as follows:
\[S = \{\bs z\in \mathcal{M}_4\,|\, \rank{\mathrm d}{\rm\Phi}<2\}.\]
The image of the critical set $S$ in the space of first
integrals ${\rm\Phi}(S)$ and the region of possible
motions ${\rm\Phi}(\mathcal{M}_4)$ form a \textit{bifurcation
diagram} of the integrable system.

The set $S$ is a union of two subsets:
\begin{gather*}S_0=\{\bs z\in \mathcal{M}_4\,|\, \rank{\mathrm d}{\rm\Phi}=0\},\\ S_1=\{\bs z\in \mathcal{M}_4\,|\, \rank{\mathrm d}{\rm\Phi}=1\}.\end{gather*}
In the case at hand the subset $S_0$ consists of separate points, and the
subset $S_1$ is a one-parameter family of closed curves. The points belonging to $S_0$ are
the fixed points of the system~\eqref{eqred}, and the set $S_1$ is filled by
periodic solutions of the system~\eqref{eqred}, which correspond
to fixed points of the reduced system~\eqref{eqsys}.

\subsection{Critical sets}

To find the sets $S_1$ or $S_0$ in redundant coordinates, it is necessary to solve the
system of equations given by the equations
\begin{equation}\left.\rank\frac{\partial(\mathcal{F}_0,\mathcal{F}_1,\mathcal{F}_2,\mathcal{E})}{\partial \bs z}\right|_{\mathcal{M}_4}=3\label{rank3}\end{equation}
or
\begin{equation}\left.\rank\frac{\partial(\mathcal{F}_0,\mathcal{F}_1,\mathcal{F}_2,\mathcal{E})}{\partial \bs z}\right|_{\mathcal{M}_4}=2,\label{rank2}\end{equation} respectively. This can be done, for example,
using the method of undetermined multipliers (see, e.g.,~\cite{borisov2013topological}).
We will not present here this procedure in explicit form since it is standard and the intermediate
expressions are cumbersome.

Next, we consider each of the critical subsets separately.

\subsubsection{The critical subset $S_1$}

Analysis of equations~\eqref{rank3} shows that they can have different
solutions, depending on whether or not the parameter $\alpha$ is equal to zero.
We first consider the general case $\alpha\neq0$.

Solving the system of equations given by~\eqref{rank3},
we can prove the following proposition.

\begin{pro}
  In the case $\alpha\neq0$ the critical subset $S_1$ of the integral map $\Phi$ has the form
  \[S_1=S_1^\alpha=\{\bs z\,|\, \bom=\om_0(\vth)\bg\x(\be_3\x\bg),\,\bg=\bg(\vth,\vfi)\},\]
  where the dependence $\bg(\vth,\vfi)$ is defined by the expressions~\eqref{newvars},
  $\om_0(\vth)$ has the form
  \begin{equation}\om_0^2=\frac{-\eta\sin^2\!\vth\left(\cos\vth(1-\beta^2)+\alpha Z(\vth)\right)}{\cos\vth Z(\vth)\left(\eta(Z(\vth)+\alpha\cos\vth)^2 + \cos^2\!\vth + \nu\sin^2\!\vth\right)},\label{usl2}\end{equation}
  $\vfi\in[0,2\pi)$, and the range of angle $\vth$ is determined from the
  condition that the expression~\eqref{usl2} be nonnegative for $\om_0^2$.
\end{pro}

The critical set $S_1^\alpha$ is filled by periodic trajectories for which the
dependences $\vth(t)$ and $\vfi(t)$ have the form
\begin{equation}
  \vth=\vth_0=\const,\q \vfi=\vfi_0+\frac{\om_0}{\sin\vth_0}t,
  \label{pr}
\end{equation}
where $\vfi_0\in[0,2\pi)$ defines the initial angle of proper rotation. Due to the axial
symmetry of the system, without loss of generality, one can set $\vfi_0=0$.
As is customary, we will call the solutions \eqref{pr} {\it permanent rotations}. In the reduced
system~\eqref{eqsys} they correspond to
a one-parameter family of fixed points
\begin{equation}\vth=\vth_0=\const,\q p_\vth=0.\label{prth}\end{equation}
In the problem we consider here, under permanent rotations, the trajectory of the ellipsoid
in the fixed coordinate system is closed and the following proposition holds.

\begin{pro}
The fixed points~\eqref{prth} correspond to the rolling of an ellipsoid in a circle with a
constant inclination angle of the symmetry axis, $\vth_0$. The point of contact
traces out a circle of radius $\rho_p=\frac{\beta^2}{Z(\vth_0)}\tan\vth_0$, and the center
of mass traces out a circle of radius $\rho_c =Z(\vth_0)\tan\vth_0+ \alpha\sin\vth_0$.
\end{pro}

\begin{proof}

    Let us parameterize the rotation matrix $\Q$ by Euler's angles: the nutation angle $\vth$,
    the precession angle $\psi$ and the angle of proper rotation $\vfi$
    (we have already introduced two of the three angles above for parameterization of the
    vector $\bg$)~\cite{dtt}
    {\small
    \begin{equation*}
    \Q\!=\!{\begin{pmatrix}
    \cos\psi\cos\vfi-\cos\vth\sin\psi\sin\vfi \!\!\!&\!\!\!\sin\psi\cos\vfi +\cos\vth\cos\psi\sin\vfi \!\!\!&\!\!\! \sin\vfi\sin\vth\\
    -\cos\psi\sin\vfi-\cos\vth\sin\psi\cos\vfi\q \!\!\!&\!\!\! -\sin\psi\sin\vfi +\cos\vth\cos\psi\cos\vfi\q\!\!\!&\!\!\! \cos\vfi\sin\vth\\
    \sin\vth\sin\psi \!\!&\!\! -\sin\vth\cos\psi\!\!&\!\! \cos\vth
    \end{pmatrix}}.
    \end{equation*}}

    The quadrature for angle $\psi$, with~\eqref{kin} and~\eqref{w12} taken into account, has
    the form
    \begin{equation}\dot\psi = -\frac{\kappa\cos\vth}{J(\vth)\sin^2\!\vth}.\label{dpsi}\end{equation}
    Thus, for the solution~\eqref{pr} the dependence $\psi(t)$ has the form
    \[\psi=\psi_0+\om_\psi t,\q \om_\psi=-\frac{\kappa\cos\vth_0}{J(\vth_0)\sin^2\!\vth_0}=\const,\]
    where $\psi_0$ is the initial value of angle $\psi$,
    which without loss of generality can be set equal to zero $\psi_0=0$.

    In accordance with~\eqref{pr}, angle $\vfi$ also changes linearly in time
    \[\vfi=\vfi_0+\om_\vfi t,\q \om_\vfi = \frac{\kappa}{J(\vth_0)\sin^2\!\vth_0}=\const,\]
    where the initial value of angle $\vfi_0$ will also be assumed to be zero.

    Substituting the resulting expressions into the quadratures~\eqref{kin} and integrating
    them over time, we obtain
    \[x_c=\rho_c\sin\psi,\q y_c = -\rho_c\cos\psi,\q z_c = \alpha\cos\vth_0+Z(\vth_0),\]
    where $\rho_c = Z(\vth_0)\tan\vth_0 + \alpha\sin\vth_0$.

   Knowing the trajectory of the center of mass $\br_c(t)$, we find the trajectory of the
   point of contact from the relation
  \begin{equation*}\br_p = \br_c + \Q\br,\label{rp}\end{equation*}
  where $\br_p=(x_p,y_p,0)$ is the radius vector of the point of contact referred to the
  fixed coordinate system.
  This yields
    \[x_p=x_c+(\br,\bal)=\rho_p\sin\psi,\q y_p=y_c+(\br,\bbt)=-\rho_p\cos\psi,\]
    where $\rho_p=\frac{\beta^2}{Z(\vth_0)}\tan\vth_0$.

\end{proof}

From the resulting expressions for the radius of the circle of the trajectory of the contact
point and the center of mass it can be seen that, as $\vth_0\rightarrow\pi/2$,
$\rho_c\rightarrow\infty$ and $\rho_p\rightarrow\infty$. That is, as the symmetry axis approaches
the horizontal position, the radius of the circle increases without bound.
Figure~\ref{pr3d} shows a schematic representation of permanent rotations of an
ellipsoid of revolution on a plane.

\begin{figure}[!ht]
\centering\includegraphics{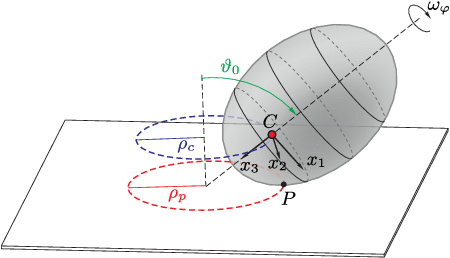}
\caption{Example of permanent rotations of an ellipsoid}
\label{pr3d}
\end{figure}

Note that, for fixed values of the parameters $\alpha$ and $\beta$, permanent rotations are
not possible for all inclination angles of the symmetry axis of the ellipsoid $\vth_0$.
Possible inclination angles are defined
by the inequality $\om_0^2(\vth)\geqslant0$. Figure~\ref{wvth} shows the regions on the
plane $(\beta^2,\vth)$ that are given by this inequality for a fixed value of $\alpha$.

\begin{figure}[!ht]
  \centering
  \includegraphics{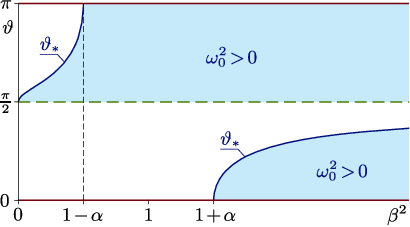}
  \caption{Regions of positive definiteness of the expression~\eqref{usl2} on the parameter
  plane $(\beta^2,\vth)$ for a fixed value of $\alpha$.}\label{wvth}
\end{figure}

\begin{rem}In~\cite{karapetyan1998families} conditions are found for
the parameters of an ellipsoid under which permanent rotations may occur at any point of the
ellipsoid's surface, and in~\cite{bizyaev2022permanent} the stability of such rotations is analyzed.\end{rem}

Analyzing the expression~\eqref{usl2} and Fig.~\ref{wvth}, we can formulate the following
statements.

\begin{utv}
  Permanent rotations under which the center of mass of the ellipsoid lies above its
  geometric center exist only for a fairly oblate ellipsoid $(\beta^2>1+\alpha)$.
\end{utv}

\begin{utv}
  Permanent rotations under which the center of mass of the ellipsoid lies below its geometric
  center exist regardless of the form of the ellipsoid. However, for a fairly
  prolate ellipsoid $(\beta^2<1-\alpha)$ such rotations are possible only for
  a fairly large inclination angle of its symmetry axis.
\end{utv}

In addition, it follows from the analysis of the expression~\eqref{usl2} and Fig.~\ref{wvth}
that there exist
inclined equilibrium positions of the ellipsoid. These equilibrium positions are
defined from the condition $\om_0^2=0$. The inclination angle for such
equilibrium positions depends on the parameters of the ellipsoid as follows:
\begin{equation}\vth_*=\arccos\frac{\alpha\beta}{\sqrt{(1-\beta^2)(1+\alpha^2-\beta^2)}}.\label{thz}\end{equation}
In Fig.~\ref{wvth}, these equilibrium positions correspond to the boundary of the region
$\om_0^2\geqslant0$. These positions exist only under some restrictions on the mass-geometric parameters of the ellipsoid. These restrictions have a simple mechanical interpretation.

\begin{utv}
  The inclined equilibrium positions of an ellipsoid of revolution with axisymmetric
  mass distribution exist only when its center of mass lies between the centers of
  curvature of its vertices.
\end{utv}

In the case $\alpha=0$ (without displacement of the center of mass), additional
symmetry leads to an increase in the critical set $S_1$.

\begin{pro}
  In the case $\alpha=0$ the critical subset $S_1$ of the integral map ${\rm\Phi}$ is a union
  of two subsets
  \begin{gather*}
    S_1=S_1^\alpha\cup S_1^0, \\
    S_1^0=\{(\bom,\bg(\vth,\vfi))\,|\,\bom=(0,0,\om_3),\,\vth=\frac{\pi}{2},\\\hspace*{5.3cm}\vfi\in[0,2\pi),\,\om_3\in \mathbb{R}\}.
  \end{gather*}
\end{pro}

The critical subset $S_1^0$ is filled by periodic trajectories
\begin{equation}\om_3=\om_0=\const,\q \vth=\frac\pi2,\q \vfi=\vfi_0+\om_0 t.\label{prpi2}\end{equation}
In the reduced system~\eqref{eqsys} all solutions~\eqref{prpi2} correspond to
one fixed point
\[\vth=\frac\pi2,\q p_\vth=0.\]
In the fixed coordinate system, the solutions~\eqref{prpi2} correspond to the rolling along the ellipsoid's equator in a straight line with an arbitrary velocity.

\subsubsection{The critical subset $S_0$}

We now consider the critical subset $S_0$. Solving the system of equations given by
equation~\eqref{rank2}, we can prove the following proposition.

\begin{pro}
  The critical subset $S_0$ of the integral map $\Phi$ consists of two points
    \[S_0=S_0^\pm=\{(\bom,\bg)\,|\,\bom=0,\,\bg=(0,0,\pm1)\}.\]
  \label{pro3}
  \end{pro}

  The critical subset $S_0$ is filled by fixed points of the system~\eqref{eqred}.
  These points correspond to the equilibrium positions of the ellipsoid in which
  it rests on one of its vertices.
  In Fig.~\ref{wvth}, the fixed points $S_0$ correspond to the horizontal straight lines $\vth=0$ and $\vth=\pi$.

  \begin{rem}
The fixed points of the complete system~\eqref{eqred} $S_0$ correspond to fixed points
of the reduced system~\eqref{eqsys}. This is due to the fact that such equilibrium positions
lie on singular zero-dimensional orbits of the action of the rotation group~\eqref{vecfield}.
\end{rem}

\subsection{Example of a bifurcation diagram}

Consider an example of constructing a bifurcation diagram which is an image of the
critical set $S$. We present the explicit form of bifurcation curves and points
corresponding to the image ${\rm\Phi}(S)$:
\begin{enumerate}
\item[--] the image of the subset $S_1^\alpha$ consists of curves on the plane of first
integrals $(\kappa,\eps)$, which can be represented as
\begin{equation}
\sigma_\vth : \left\{
    \begin{aligned}
    &\eps = \frac{3Z(\vth_0)^2-1}{2Z(\vth_0)} + \alpha\frac{3\cos^2\!\vth_0-1}{2\cos\vth_0},\\
    &\kappa= \pm\sin^2\!\vth_0\sqrt{\frac{\beta^2-1}{Z(\vth_0)} - \frac{\alpha }{\cos\vth_0}},
    \end{aligned}
    \right.
    \label{s1}\end{equation}
    where
    \[\begin{aligned}
    &\vth_0\in(\pi/2,\vth_*)\q\text{for}\q \beta^2<1-\alpha,\\
    &\vth_0\in(\pi/2,\pi)\q\text{for}\q 1-\alpha<\beta^2<1+\alpha,\\
    &\vth_0\in(0,\vth_*]\cup(\pi/2,\pi)\q\text{for}\q \beta^2>1+\alpha;
    \end{aligned}\]
\item[--] the image of the subset $S_1^0$ is a parabola
\begin{equation}\sigma_{\pi/2} : \,
    \eps = \frac{\kappa^2}{2}+\beta;
    \label{s2}\end{equation}
\item[--] the image of the subset $S_0$ on the plane of first integrals $(\kappa,\eps)$
consists of a pair of point
\begin{equation}\sigma_0\,:\,(0,1+\alpha),\qq \sigma_\pi\,:\,(0,1-\alpha).\label{s0}\end{equation}
Depending on the values of the parameters $\alpha$ and $\beta$, these points can be isolated.
It can be seen from Fig.~\ref{wvth} that the point $\sigma_0$ is isolated when
$\beta^2<1+\alpha$, and the point $\sigma_\pi$, when $\beta^2<1-\alpha$.
We note that in the case $\alpha=0$ these points lie
on the same level set
of the integral $\kappa=0$, $\eps=1$, but do not coincide in $\mathcal{M}_6$.
\end{enumerate}

Figure~\ref{bifex} shows an example of constructing
a bifurcation diagram on the plane of first integrals $(\kappa,\eps)$ for the general case
$\alpha\neq0$.
This example of a bifurcation diagram is given for the case
of an ``oblate'' ellipsoid $\beta^2>1$ with a relatively small displacement of the center of
mass.
In this section, we explain the notation used in the diagram, and in the following sections
we will use them without repeating our explanations.

\begin{figure}[!ht]
\centering\includegraphics{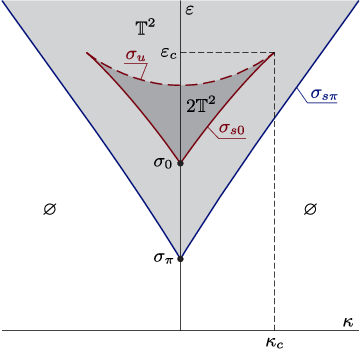}
\caption{Example of a bifurcation diagram. This type of diagram is characteristic of
an ``oblate'' ellipsoid with a fairly small displacement of the center of mass
$\alpha$.\label{bifex}}
\end{figure}

In the diagram, gray denotes the region of possible motions of the system.
In the light gray region, the integral manifold is two-dimensional torus $\mathbb{T}^2$.
Dark gray denotes the region which
corresponds to two connectedness components of the
integral manifold (two two-dimensional tori $2\mathbb{T}^2$). These two connectedness
components are glued together to each other along the curve $\sigma_u$. In the white region,
no motions of the system are possible ($\varnothing$).

The heavy solid and dotted lines in the diagram show parts of the
bifurcation curves~\eqref{s1} (and~\eqref{s2} for the case $\alpha=0$). The solid
curves correspond to stable fixed points, and the dotted curves indicate unstable ones.
The curve~\eqref{s1} is partitioned by critical values of
angle $\vth_0$ into separate segments with the same type of
stability, which are denoted in the diagram by different symbols.  The curve $\sigma_{s\pi}$
corresponds to
$\vth_0\in(\pi/2,\pi)$,
the curve $\sigma_{s0}$ corresponds to $\vth_0\in(0,\vth_c)$, and the
curve $\sigma_{u}$, to $\vth_0\in(\vth_c,\vth_*]$.
 Here $\vth_c$ is the critical value of angle $\vth_0$. At this value, $\vth_0=\vth_c$ (which corresponds to the point
$(\kappa_c,\eps_c)$ in the diagram), one can observe a change in the type of stability
of fixed points corresponding to the curve
$\sigma_\vth$. The fixed points~\eqref{s0} corresponding to equilibrium  positions are
shown as colored black points (in the case of their stability) or punctured points
(in the case of their instability).

As the parameters of the system $\alpha$, $\beta$ change, the critical subsets $S_1$ and $S_0$
can change, resulting in a qualitative change in the type of the diagram. The
bifurcations of critical points described above occur simultaneously with this change. In order to determine possible types of
bifurcation diagrams, we investigate the bifurcations of critical points in what follows.

\subsection{Bifurcations of critical sets and types of~bifurcation diagrams}
\subsubsection{Bifurcations of the critical subset $S_0$}

Consider the bifurcations of the critical subset $S_0$.

When $\beta^2<1-\alpha$ (see Fig.~\ref{wvth}), the subset $S_0$ consists of two isolated points.
When $\beta^2=1-\alpha$, a bifurcation occurs, resulting in the component $S_0^-$ merging with
the component $S_1^\alpha$.

Next, when $1-\alpha<\beta^2<1+\alpha$, one of the points of the subset $S_0$ is isolated
($S_0^+$) and the other is not isolated ($S_0^-$).

When $\beta^2=1+\alpha$, a second bifurcation occurs during which
the component $S_1^\alpha$ is born from $S_0^+$, and both points of the subset $S_0$ are not
isolated now.

We note that, simultaneously with the bifurcations of critical sets described above,
changes in the
stability of the fixed points occur. To illustrate this fact, we analyze the
linear stability of the fixed points
\begin{equation}
\vth=0,\q p_\vth=0\qq \text{and}\qq \vth=\pi,\q p_\vth=0
\label{pr0}
\end{equation}
of the reduced system~\eqref{eqsys}.

To do so, we represent the system~\eqref{eqsys} as
\[\dot{\bs q}={\bf f}(\bs q),\]
where $\bs q = (\vth,p_\vth)$, ${\bf f}(\bs q)$ is the vector function whose elements depend
on the components of $\bs q$. Introducing the perturbation vector $\bs{\tilde q}=\bs q - \bs q_0$,
where $\bs q_0 = (\vth_0, 0)$ is the equilibrium position, and representing the equation for
$\bs{\tilde q}$ as a series expansion near the equilibrium position, we obtain
\[\dot{\tilde{\bs q}} = {\bf \Lambda}\bs{\tilde q},\q {\bf \Lambda} = \left.\frac{\partial {\bf f}(\bs q)}{\partial \bs q}\right|_{\bs q = \bs q_0}.\]
The matrix $\bf \Lambda$ has the form
\[{\bf \Lambda} = \begin{pmatrix}
                    0 & 1 \\
                    \frac{G_0'(\vth_0)}{B(\vth_0)} & 0
                  \end{pmatrix},\]
where the notation $G_0(\vth_0)=\left.G(\vth_0)\right|_{p_\vth=0}$ is used, i.e.,
$G_0'(\vth_0)$ denotes the derivative of the function $G(\vth)$ with respect to $\vth$,
calculated at the point $\vth_0$ for $p_\vth=0$.

We write the characteristic equation
\[\det({\bf \Lambda - \lambda\E_2})=0,\]
where $\lambda$ are the eigenvalues of the matrix $\bf \Lambda$ and $\E_2$ is
the $2\x2$ identity matrix. This equation
reduces to the simple quadratic equation
\[\lambda^2-\frac{G_0'(\vth_0)}{B(\vth_0)}=0.\]

Since $B(\vth_0)>0$ for any values of $\vth_0$, the eigenvalues of $\lm$ are defined by
the value of $G_0'(\vth_0)$. If $G_0'(\vth_0)<0$, the corresponding fixed point is a stable
center-type point. If $G_0'(\vth_0)>0$, then the corresponding fixed point is an unstable
saddle point.

For the lower equilibrium position ($\theta_0=\pi$, $\kappa=0$) the value of $G_0'(\vth_0)$ is negative
for $\beta^2>1-\alpha$, and for the upper equilibrium position ($\theta_0=0$, $\kappa=0$) it is negative for $\beta^2>1+\alpha$.
Thus, depending on the type of stability of the fixed points~\eqref{pr0}, the parameter plane
$(\alpha,\beta^2)$ can be divided into three regions:
\begin{enumerate}
  \item[(a)] $\beta^2<1-\alpha$. In this region, both fixed points $\sigma_0$, $\sigma_\pi$
  are unstable and isolated;
  \item[(b)] $1-\alpha<\beta^2<1+\alpha$. In this region, the lower position of the ellipsoid
  $\sigma_\pi$ is stable and is nonisolated, and the upper one, $\sigma_0$, is unstable and
  isolated;
  \item[(c)] $\beta^2>1-\alpha$. In this region, both fixed points are stable $\sigma_0$,
  $\sigma_\pi$ and nonisolated.
\end{enumerate}

\begin{rem}
We note that the resulting boundaries of the stability of the solutions $\sigma_0$ and
$\sigma_\pi$ are defined only by the geometry of the body and the position of the center of
mass. For the stability of the vertical position of the ellipsoid it is necessary
that its center of mass ($1\pm\alpha$) be below the center of curvature ($\beta^2$)
of the ellipsoid's surface at the point of contact. The center of curvature in this case
is correctly determined because the body is axisymmetric and the principal radii of curvature
at the vertices of the ellipsoid are the same.
This statement is in agreement with the results obtained in the book~\cite{markeev1992dynamics} for
an arbitrary body of revolution and in our recent paper~\cite{kilin2023stability} on the dynamics of an ellipsoid on
a vibrating plane.
\end{rem}

\subsubsection{Bifurcations of critical subsets $S_1$}

We now turn our attention to bifurcations of the critical subset $S_1$, which
corresponds to the fixed points~\eqref{prth} of the reduced system~\eqref{eqsys}. The
topological type of the subset $S_1$ is uniquely related to the type of the bifurcation
curve~\eqref{s1}. The type of bifurcation
curve is determined, in its turn, by its critical points, which satisfy the
system of equations
\begin{equation}G_0(\vth)=0,\q\dfrac{\partial G_0(\vth)}{\partial\vth}=0.\label{eq2}\end{equation}
In the bifurcation diagram, these critical points correspond
to the cusp points and tangency points of
the bifurcation curves; also, the bifurcation curves can lose smoothness at these points.

Let us analyze the bifurcations of the critical points depending on the values of the
parameters $\alpha$ and $\beta$.

The emergence of new critical points is defined by the additional condition
\begin{equation}\dfrac{\partial^2 G_0(\vth)}{\partial\vth^2}=0.\label{eq3}\end{equation}
The common solution of the system of equations~\eqref{eq2}--\eqref{eq3}
defines the curve on the plane $(\alpha,\beta^2)$ which
separates regions with different numbers of critical points and hence
with different types of bifurcation diagrams.

In the case $\alpha\neq0$, the system of equations~\eqref{eq2}--\eqref{eq3}  gives two solutions:
\begin{enumerate}
  \item $\beta^2=1+\alpha$. This straight line on the plane $(\alpha,\beta^2)$ separates
  the regions with one cusp point (with $\beta^2>1+\alpha$) and without cusp points (with
  $\beta^2<1+\alpha$). In addition, as shown in the preceding section,
  when this straight line is intersected, a bifurcation curve $\sigma_\vth$ (branches $\sigma_{s0}$
  and $\sigma_u$) is born from point $\sigma_0$, and $\sigma_0$ ceases to be isolated.
  \item $\beta^2=1-\alpha$.  This straight line on the plane $(\alpha,\beta^2)$ corresponds
  to the loss of smoothness of the bifurcation curve $\sigma_{s\pi}$ with $\kappa=0$,
      in which case the point $\sigma_\pi$ descends  to the
      bifurcation curve $\sigma_{s\pi}$ and ceases to be isolated.
\end{enumerate}

As we see, the bifurcations of $S_1$ listed above occur at the same time as those of $S_0$.
This system has no other bifurcations of $S_1$ which do not affect $S_0$.

As was shown above, when $\alpha=0$ the subset $S_1$ bifurcates. Next, we consider the question
of bifurcations of $S_1$ for constant $\alpha=0$ and for the varying parameter $\beta$.
In the case $\alpha=0$ the system of equations~\eqref{eq2}--\eqref{eq3} has a unique
solution $\beta^2=1$. For $\beta^2<1$ there exists only one curve
$\sigma_{\pi/2}$ on the plane of first integrals, and the fixed points $\sigma_0$ and $\sigma_\pi$
are isolated. When $\beta^2=1$, bifurcation curves
$\sigma_{s0}$ and $\sigma_{s\pi}$ are born from the points
$\sigma_0$ and $\sigma_\pi$, respectively. The points
$\sigma_0$ and $\sigma_\pi$ cease to be isolated, and on the
curve $\sigma_{\pi/2}$ tangency points are born at which the
curves  $\sigma_{s0}$ and $\sigma_{s\pi}$ are tangent to the
curve $\sigma_{\pi/2}$.

Simultaneously, the stability of the fixed
point corresponding to the
curve $\sigma_{\pi/2}$ changes. Indeed, following the
above-mentioned algorithm of the
linear stability analysis of fixed points, the condition
for stability is defined by the sign $G_0'(\pi/2)$,
whence we find that the point~\eqref{prpi2} is stable
for
\begin{equation}|\kappa|<\kappa_c=\sqrt{\frac{\beta^2-1}{\beta}}.\label{a0usl}\end{equation}
From this, we obtain the critical value of the parameter
\[\beta^2=1.\]
Thus, when $\beta^2<1$ the fixed point \[\vth=\frac\pi2,\q p_\vth=0\] is stable
for any values of $\kappa$, and for $\beta^2>1$ it is stable
only when condition~\eqref{a0usl} is satisfied.

\subsubsection{Types of bifurcation diagrams}
Summarizing, we note that the entire parameter space $(\alpha,\beta^2)$ can be divided into
five regions corresponding to different types of bifurcation diagrams (see Fig.~\ref{ab}).
Figure~\ref{bif} shows all possible types of bifurcation diagrams for the problem under
consideration. Diagrams (a)--(c) correspond to the general case $\alpha>0$ and
have been constructed for the fixed value $\alpha=0.5$ and different values of $\beta$, and
diagrams (c) and (d), for the case $\alpha=0$.

\begin{figure}[!ht]
  \centering
  \includegraphics{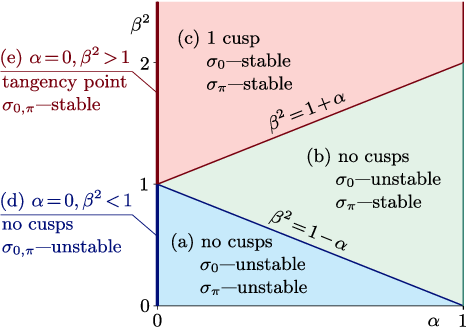}
  \caption{Regions with different types of bifurcation diagrams on the parameter plane $(\alpha,\beta^2)$.\label{ab}}
\end{figure}

\begin{figure}[!ht]
  \centering
  \includegraphics[width=\textwidth]{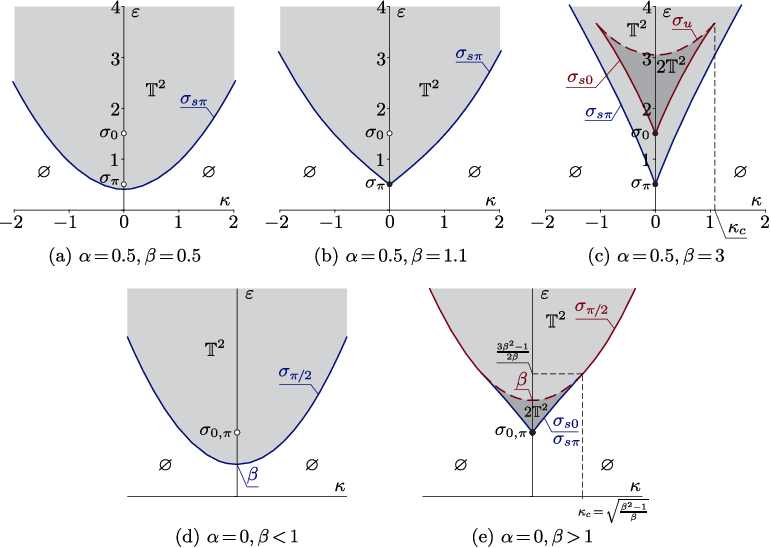}
  \caption{Possible types of bifurcation diagrams for (a)--(c) $\alpha=0.5$ and (d), (e) $\alpha=0$.
  \label{bif}}
\end{figure}

Next, we consider in more detail each of the types of diagrams. In the names of the types of
diagrams (a)--(e) we refer to Fig.~\ref{bif}. In addition, for each type of diagrams we
construct all possible types of phase portraits of the reduced system~\eqref{eqsys}
on different level sets of the integral $\mathcal{F}_2$.

\subsubsection{Bifurcation diagram (a)}

Figure~\ref{bifa} shows a bifurcation diagram and
the corresponding  phase portraits for the case where the bifurcation diagram has no cusp points
for $\alpha\neq0$. As can be seen from this figure, both points $\sigma_0$ and $\sigma_\pi$
do not lie on the bifurcation curve $\sigma_\vth$ and are
isolated.

\begin{figure}[!ht]
  \centering
  \includegraphics[width=\textwidth]{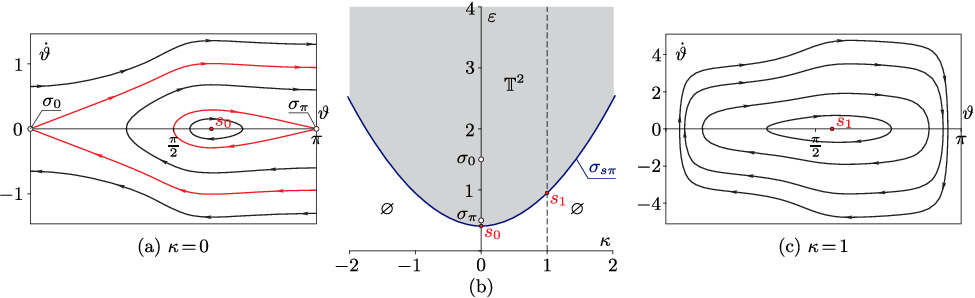}
  \caption{Bifurcation diagram (b) and examples of
  phase portraits (a), (c) corresponding to different values of the integral $\kappa$
  for $\alpha=0.5$, $\beta=0.5$.\label{bifa}}
\end{figure}

For this case, two types of qualitatively different
phase portraits are possible: for $\kappa=0$ and for $|\kappa|>0$. In Fig.~\ref{bifa},
the points $s_0$ and $s_1$ denote fixed points of the family $\sigma_\vth$ which lie on
the chosen level sets of the first integral $\kappa=0$ and $\kappa=1$,
respectively.

As can be seen from Fig.~\ref{bifa}a, for the case $\kappa=0$ the phase trajectories pass
through the poles $\vth=0$ and $\vth=\pi$. For a correct passage through the pole, it is
necessary to glue together the trajectories in the phase portrait:
$\dot\vth\rightarrow-\dot\vth$. Thus, in this case each level set of the first integrals
corresponds to one connectedness component, which is in complete agreement with the
bifurcation diagram.

\subsubsection{Bifurcation diagram (b)}

As the parameter $\beta$ increases, the curve $\sigma_{s\pi}$ approaches the point
$\sigma_\pi$. When $\beta^2=1-\alpha$, the point $\sigma_\pi$ descends to the bifurcation
curve $\sigma_{s\pi}$. Also, $\sigma_\pi$ ceases to be isolated, and the branch $\sigma_{s\pi}$
loses smoothness at this point ($\kappa=0$). The corresponding type of diagram and
the phase portraits are presented in Fig.~\ref{bifb}.

\begin{rem}
A similar bifurcation, not when the parameter, but the level set of the
integrals is changed, can be observed in the classical problem of the Lagrange top. The corresponding
conditions for stability of a sleeping top are called the Maievsky condition~\cite{dtt,markeev1992dynamics}.
\end{rem}

\begin{figure}[!ht]
  \centering
  \includegraphics[width=\textwidth]{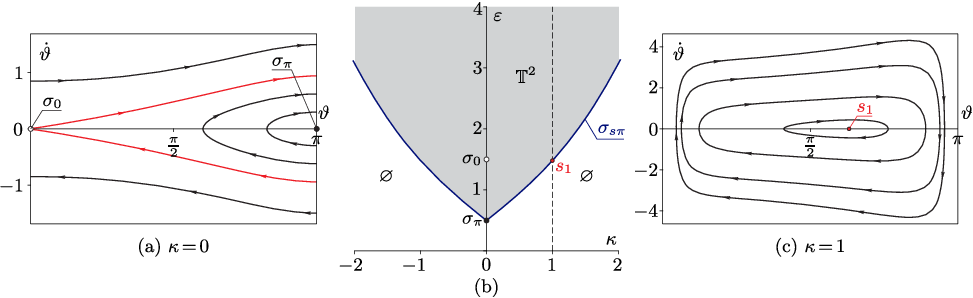}
  \caption{Bifurcation diagram (b) and examples of
  phase portraits (a), (c) corresponding to different values of the integral $\kappa$ for $\alpha=0.5$, $\beta=1.1$.\label{bifb}}
\end{figure}

For this case also, only two types of
qualitatively different phase portraits similar to diagram (a) are possible. The only
difference is that the point $\sigma_\pi$ is stable for $\kappa=0$.

\subsubsection{Bifurcation diagram (c)}

Further increase in the parameter $\beta$, with $\beta^2=1+\alpha$, leads to a new bifurcation.
The second isolated point $\sigma_0$ gives rise to bifurcation branches $\sigma_{s0}$ and
$\sigma_u$. The corresponding diagram and the phase portraits are presented in Fig.~\ref{php}.

\begin{figure}[!ht]
  \centering
  \includegraphics[width=\textwidth]{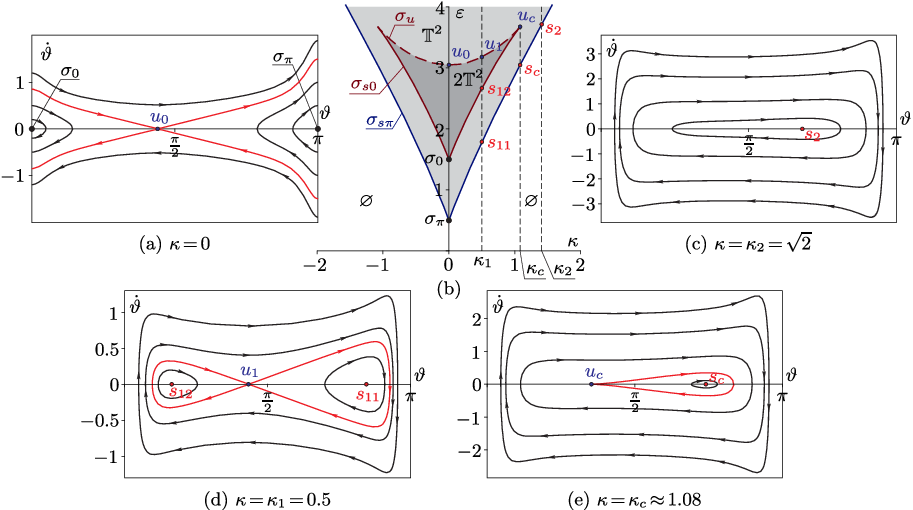}
  \caption{Bifurcation diagram (b) and examples of
  phase portraits (a), (c)--(e) corresponding to different values of the integral $\kappa$
  for $\alpha=0.5$, $\beta=3$.\label{php}}
\end{figure}

As can be seen from the figure, in this case there exist four different types of
phase portraits. When $\kappa=0$ and $0<|\kappa|<\kappa_c$ (Figs.~\ref{php}a, \ref{php}d)
the phase portrait has three fixed points: two of them are stable ($\sigma_0$, $\sigma_\pi$
for $\kappa=0$ and $s_{11}$,
$s_{12}$ for $\kappa=\kappa_1$) and one is unstable ($u_0$ and $u_1$, respectively).
When $\kappa=\kappa_c$ (Fig.~\ref{php}e) the curves $\sigma_u$ and $\sigma_{s0}$ merge
at the point $u_c$. In this case, the phase portrait exhibits a saddle-node
bifurcation --- the stable point and the unstable point merge and disappear.
When $|\kappa|>\kappa_c$ (Fig.~\ref{php}c), only one fixed point remains in the phase
portrait.

\goodbreak

\subsubsection{Bifurcation diagram (d)}

In the case $\alpha=0$, $\beta<1$ the bifurcation diagram consists of
the curve $\sigma_{\pi/2}$, which corresponds to rolling along the ellipsoid's equator, and of
two isolated fixed points $\sigma_0$ and $\sigma_\pi$, which lie on the same level set of the
integrals.

\begin{figure}[!ht]
  \centering
  \includegraphics[width=\textwidth]{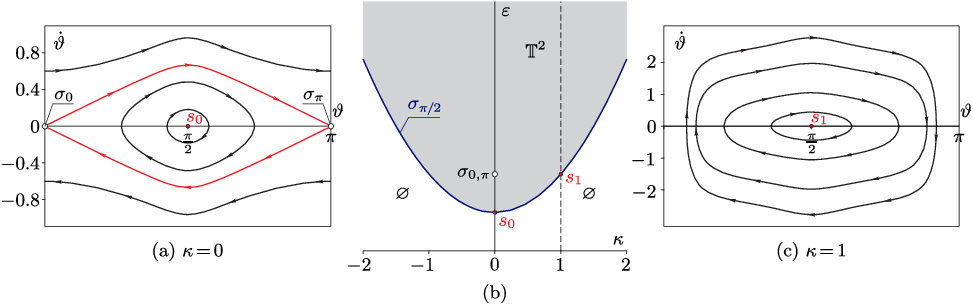}
  \caption{Bifurcation diagram (b) and examples of
  phase portraits (a), (c) corresponding to different values of the integral $\kappa$ for
  $\alpha=0$, $\beta=0.5$.\label{bifd}}
\end{figure}

The phase portrait has, for any value of the integral $\kappa$, a fixed point for
$\theta=\pi/2$, and is symmetric about this point.
For $\beta<1$, two qualitatively different types of phase portrait are possible
(see Fig.~\ref{bifd}). When $\kappa=0$, the phase plane has three fixed points: two
unstable ones, $\sigma_0$ and $\sigma_\pi$, and a stable fixed point, $s_0$ (which belongs to
the family $\sigma_{\pi/2}$). When $\kappa>0$, only one stable fixed point, $s_1$,
which belongs to the family $\sigma_{\pi/2}$, remains in the phase portrait.

\subsubsection{Bifurcation diagram (e)}

When $\beta=1$ (with constant $\alpha=0$), the fixed points $\sigma_0$ and $\sigma_\pi$
corresponding to the vertical positions of the ellipsoid descend to the bifurcation curve $\sigma_{\pi/2}$ and give rise to bifurcation branches
$\sigma_{s0}$ and $\sigma_{s\pi}$, respectively. In this case, they become
stable and cease to be isolated (see Fig.~\ref{bife}). 
unstable for $\kappa=0$.

\begin{figure}[!ht]
  \centering
  \includegraphics[width=\textwidth]{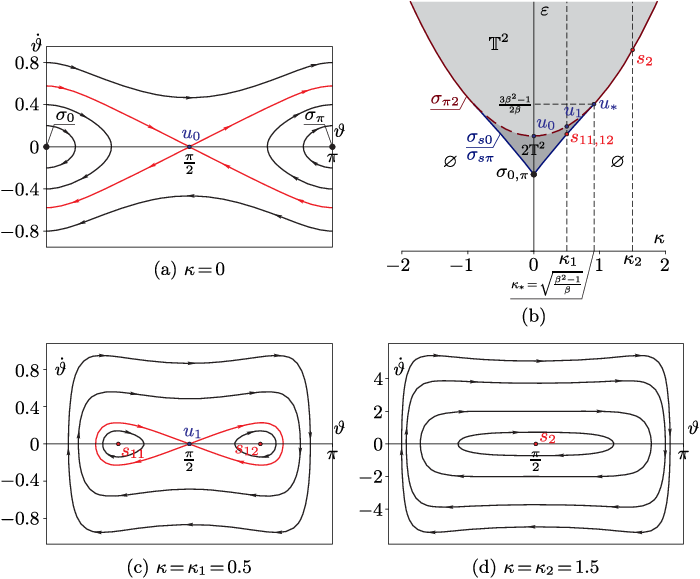}
  \caption{Bifurcation diagram (b) and examples of phase portraits (a), (c), (d)
  corresponding to different values of the integral $\kappa$ for $\alpha=0$, $\beta=1.5$.\label{bife}}
\end{figure}

When $\kappa=0$, the phase portrait has three fixed points: two stable ones, $\sigma_0$ and $\sigma_\pi$, and an unstable one, $u_0$, at $\vth=\pi/2$.

When $0<|\kappa|<\kappa_c$, the phase portrait has three fixed points: two stable ones and an
unstable one. Next, when $|\kappa|=\kappa_c$, a pitch-fork bifurcation occurs, such that
the three fixed points merge into a stable fixed point\footnote{We do not present here the phase
portrait at the moment of bifurcation for $\kappa=\kappa_c$ since visually it does not differ
from the case $\kappa>\kappa_c$.}. As a result, when $|\kappa|>\kappa_c$, only one stable
point ($\vth=\pi/2$) remains in the phase portrait.

\section{Analysis of the trajectories of an ellipsoid\label{sec5}}

In this section, we analyze how an ellipsoid behaves during motion on a plane,
depending on the initial conditions and parameters of the problem, that is,
to what motions of the ellipsoid the periodic trajectories of the reduced
system~\eqref{eqsys} correspond.

We have already ascertained that permanent rotations correspond to the rolling motion of the
ellipsoid (both its center of mass and the point of contact) in a circle. In the
limiting cases when $\vth=0$ or $\vth=\pi$, the circle
degenerates to a point, and as
$\vth\rightarrow\pi/2$, it degenerates to a straight line. Let us find characteristic
trajectories of the ellipsoid for other solutions.

In Section~\ref{oned} we have obtained the system with one degree of freedom~\eqref{eqsys}, which admits
an energy integral. In the general position case, the solution of this system is
periodic motion with frequency $\om_\vth$ along closed trajectories.
From the expressions for the quadrature of the angle $\psi$~\eqref{dpsi} it follows that
the function $\dot\psi$ is also periodic with the same frequency $\om_\vth$.
Following our previous work~\cite{kilin2017rolling,kilin2019qualitative}, we represent
the right-hand side of the quadrature for angle $\psi$~\eqref{dpsi} as a
Fourier series expansion
\begin{equation}
\dot\psi = \om_\psi + \sum\limits_{n\in\mathbb{Z}\backslash\{0\}} {\rm\Psi}_n{\rm e}^{\j n\om_\vth t},\label{fourier}\end{equation}
where ${\rm\Psi}_n$ are the coefficients of Fourier series expansion, $\j$ is an imaginary unit, and
$\om_\psi$ is the zeroth term of the expansion which is calculated by the formula
\[
\om_\psi = \frac1{T_\vth}\int\limits_0^{T_\vth}\dot\psi(t)dt,\q T_\vth=\frac{2\pi}{\om_\vth}.\]
Integrating the expression~\eqref{fourier} over time, we can represent the expression for
$\psi(t)$ as
\begin{equation}
\psi(t) = \om_\psi t + \psi_\vth(t),\q \psi_\vth(t) = \sum\limits_{n\in\mathbb{Z}\backslash\{0\}} \frac{{\rm\Psi}_n}{\j n\om_\vth} {\rm e}^{\j n\om_\vth t}.\label{pp}\end{equation}
Here $\psi_\vth(t)$ is also a periodic function with frequency $\om_\vth$.

Next, we analyze the trajectory of the center of mass of the ellipsoid. We introduce
the complex variable $\zeta=x_c+\j y_c$, where  $x_c$ and $y_c$ are the coordinates of
the center of mass of the ellipsoid.
Using~\eqref{kin}, we can write the equation for changing the variable $\zeta$ in the form
\begin{gather*}\dot\zeta = \frac{\kappa}{J(\vth)\sin\vth}(Z(\vth)+\alpha\cos\vth){\rm e}^{\j\psi(t)}.\end{gather*}
In view of~\eqref{pp} and the Fourier series expansion of $\dot\zeta$, the expression for
$\zeta$ has the form
\[\zeta = \zeta(0)+\sum\limits_{n\in\mathbb{Z}}\frac{v_n}{\j(\om_\psi+n\om_\vth)}{\rm e}^{\j(\om_\psi+n\om_\vth)t},\]
where $v_n$ are the coefficients of the Fourier series expansion of $\dot\zeta$. Thus, in the general case the
trajectory of the center of mass of the ellipsoid is a bounded curve except for
the cases where the resonance relation $\om_\psi+n\om_\vth=0$ is satisfied.

Next, we introduce the concept of rotation number $N$
as a ratio between frequencies $\om_\psi$ and $\om_\vth$
\[N=-\frac{\om_\psi}{\om_\vth}=-\frac1{2\pi}\int\limits_0^{T_\vth}\dot\psi(t)dt.\]
The value of the rotation number determines the type of trajectory of the
ellipsoid in absolute space.
To each trajectory of the system~\eqref{eqsys} there corresponds
its own rotation number depending on the values of frequencies
$\om_\vth$ and $\om_\psi$.
The frequencies
$\om_\psi$ and $\om_\vth$, in their turn, are uniquely defined by the values of the first integrals
$\kappa$, $\eps$ and by the connectedness component of the integral manifold
$\mathcal{M}_{\kappa,\eps}$, on which the trajectory lies. Thus, the rotation number $N$
can be defined as a function given in the (many-sheeted) bifurcation diagram. Hence, the type of
trajectory of the ellipsoid in absolute space is uniquely (up to the chosen sheet) defined
by the values of the first integrals $\kappa$ and $\eps$.

We now turn to the question of possible types of trajectories of the ellipsoid depending on $\kappa$ and
$\eps$.

We note that it follows from~\eqref{dpsi} that, when $\kappa=0$, one has $N=0$ for any parameter
values and initial conditions. As a result, the integral manifolds $\mathcal{M}_{0,\eps}$
are completely filled by periodic solutions, i.e., they are families of resonant tori.
The motions on these tori can be of two types. For small energies $\eps<\eps_{\rm min}$ they
are oscillatory motions near an equilibrium position. For large energies
$\eps>\eps_{\rm min}$ the ellipsoid rolls to infinity along the meridian crossing both
of its vertices. The value of energy $\eps_{\rm min}$ is determined from the condition
\[\eps_{\rm min} = \left\{\!
      \begin{aligned}
      &\left.\eps\right|_{\sigma_0}=1+\alpha&\text{for }\, \beta^2<1+\alpha,\\
      &\left.\eps\right|_{\sigma_u,\kappa=0}=\sqrt{\frac{1+\alpha^2-\beta^2}{1-\beta^2}}\beta&\text{for }\, \beta^2>1+\alpha.
      \end{aligned}
      \right.\]
The critical invariant manifolds corresponding to the intersection of $\kappa=0$ with
bifurcation curves are filled by inclined equilibrium positions~\eqref{thz} and,
in the case of instability, by trajectories asymptotic to them.

Analyzing the resulting dependences, we can formulate the
following proposition as a conclusion from the above reasoning
 (see also~\cite{disk,borisov2017inhomogeneous}).

\begin{pro}
The trajectory of the center of mass of the ellipsoid of revolution on a plane can be
of one of the following types:
\begin{enumerate}
\item \textbf{a point} in the case of equilibrium positions $\sigma_0$, $\sigma_\pi$ and
    an inclined equilibrium position for $\kappa=0$, $\eps=U(\vth_*)$.
  \item \textbf{a circle} in the case of permanent rotations~\eqref{pr}
  with a constant inclination angle $\vth_0=\const$ $($see
  Fig.~\ref{pr3d}$)$.
  \item \textbf{an unbounded straight line} in the case of permanent rotations $\sigma_{\pi/2}$
  for $\alpha=0$ or in the case of rolling along the meridian for $\kappa=0$ and $\eps>\eps_{\rm min}$.
  \item \textbf{a straight-line segment} in the case of pendulum motions near
  an equilibrium position along the meridian for $\kappa=0$ and $\eps<\eps_{\rm min}$
  and in the case of trajectories asymptotic to unstable equilibrium positions.
  \item \textbf{a trajectory going to infinity without bound} for integer
  rotation numbers $N\in\mathbb{Z}$ $($see Fig.~\ref{res}$)$.
  \item \textbf{a closed periodic trajectory} for fractionally
  rational rotation numbers $N\in\mathbb{Q}\backslash\mathbb{Z}$ $($see Fig.~\ref{xy}a$)$.
  \item \textbf{a bounded quasi-periodic curve} for all
  other trajectories of the system~\eqref{eqsys},
  except for singular trajectories asymptotic to unstable
  fixed points $($sepa\-ra\-tri\-ces$)$ $($see Fig.~\ref{xy}b$)$.
  \item \textbf{a bounded curve asymptotically tending in forward and backward time
  to two circles}; these trajectories of the ellipsoid
  correspond to the trajectories of the reduced system~\eqref{eqsys}
   which are asymptotic to unstable permanent rotations
   $\sigma_u$.
   \item \textbf{an unbounded curve asymptotically tending in forward and backward time to two straight lines};
  these trajectories of the ellipsoid correspond to the
  trajectories of the reduced system~\eqref{eqsys}
  which are asymptotic to unstable equilibrium positions
  $\sigma_0$.
\end{enumerate}
\end{pro}

Since the rotation number $N$ depends on the values of the
first integrals $\kappa$ and $\eps$, the conditions of resonance
define some resonance curves on the plane of first integrals $(\kappa,\eps)$. Examples of such curves for a symmetrically
truncated sphere with a displaced center of mass were constructed
in~\cite{kilin2019qualitative}.

Figure~\ref{res} shows an enlarged fragment of the bifurcation diagram for $\alpha=0.5$ and
$\beta=3$ (see Fig.~\ref{bif}c). The light blue curve corresponding to the resonance $N=0$
lies near the
curve $\sigma_u$ and ``goes round the curve'' as it were. But since the lower part of the
curve runs very close to $\sigma_u$, it seems --- when one looks at the figure --- as if they
coincide. Higher-order resonances $N>0$ lie still closer to $\sigma_u$, and so we do not
represent them here.
Also, Fig.~\ref{res} shows the trajectory of the center of mass of the ellipsoid
(the dark red solid curve in the left part of Fig.~\ref{res}) and of its point of contact
(dotted dark blue curve) for
the initial conditions, which correspond to the resonance $N=0$. It can be seen that, for
these initial conditions, the ellipsoid does not, on average, drift along the axis $Oy$,
but goes to infinity along the axis $Ox$.

\begin{figure}[!ht]
\centering\includegraphics[width=\textwidth]{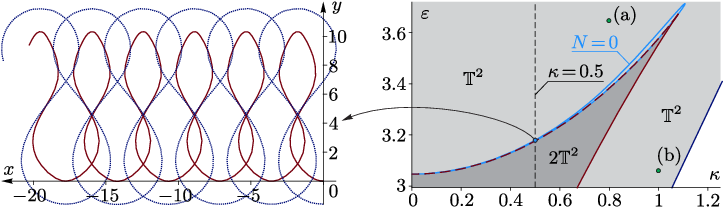}
\caption{Right: an enlarged fragment of the bifurcation
diagram (Fig.~\ref{bif}c, $\alpha=0.5$, $\beta=3$) with the curve corresponding to
resonance $N=0$.
Left: the trajectory of the center of mass (solid curve) and of the point of contact
(dotted curve) of the ellipsoid for the case where the resonance relation $N=0$ is satisfied
and for
$\kappa=0.5$.\label{res}}
\end{figure}

The numerical analysis has enabled us to find integer-valued resonances only near the
curve corresponding to unstable fixed points ($\sigma_u$), i.e., only for the case of
diagram (c) in Fig.~\ref{bif}. Based on the results obtained, we can formulate the following
hypothesis.

\begin{hyp}
The trajectories of an ellipsoid of revolution on a plane are bounded for
\begin{enumerate}
\item[$\bullet$] $\beta^2<1+\alpha$ except for the case $\kappa=0$ at $\eps>\eps_{\rm min}$;
\item[$\bullet$] $\beta^2>1+\alpha$, $\alpha>0$ and $\kappa>\kappa_{\rm max}$, where $\kappa_{\rm max}$ is a
solution of the system
\[N(\kappa,\eps)=0,\q \frac{\partial N}{\partial \eps}=0.\]
 \end{enumerate}
\end{hyp}

Figure~\ref{xy} shows the trajectories of the center of mass and of the point of contact
of the ellipsoid at resonance $N=-1/7$ (Fig.~\ref{xy}a) and in the case of a quasi-periodic trajectory (Fig.~\ref{xy}b). The values
of the first integrals for which these trajectories are constructed are shown in Fig.~\ref{res}
as points (a) and (b), respectively. Both of these cases correspond to bounded motion of
the ellipsoid on a plane.

\begin{figure}[!h]
\centering\includegraphics[width=\textwidth]{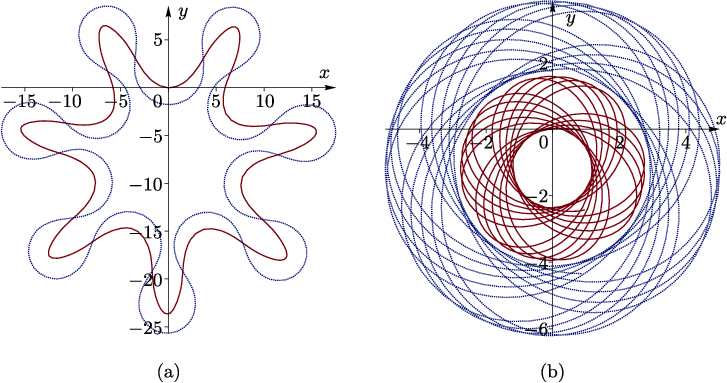}
\caption{(a) Example of a periodic trajectory (the case $N=-1/7$) for $\kappa=0.8$ and
$\vth(0)=0.4678$. (b)
Example of a quasi-periodic trajectory for $\kappa=1$ and $\vth(0)=2.1$. The solid
dark red curve denotes the trajectories of the center of mass, and the
dotted dark blue curve indicates the trajectories of the
point of contact. Both parts of Fig.~\ref{xy} are constructed for $\alpha=0.5$, $\beta=3$, $\nu=0.5$, $\eta=0.5$ and  $\dot\vth(0)=0$.\label{xy}}
\end{figure}

\section*{Summary}

In this work, we have carried out a complete qualitative analysis of the dynamics of an
ellipsoid of revolution with a displaced center of mass on a plane with no-slip and
no-spin conditions.
As is well known~\cite{borisov2008conservation,borisov2013hierarchy}, the problem of a body of revolution
rolling on a plane is integrable for the above-mentioned constraints and, moreover, all
integrals are represented in elementary functions. For an ellipsoid of revolution, we have found
particular solutions corresponding to the rolling motion in a circle. Using these solutions,
we have constructed bifurcation diagrams and performed a complete classification of these diagrams depending
on the parameters of the problem.

In addition, we have shown that in the general case the motion of the ellipsoid of revolution
occurs along a quasi-periodic bounded curve. However, depending on the values of the first integrals,
there can also exist both unbounded and closed periodic trajectories.

A possible avenue for future research is to simulate the motion of a mobile robot
of elliptic form by adding internal propulsion devices, for example, gyrostats or a
pendulum/a platform or their combination. In addition,
to simulate the motion of such a robot under conditions close to experimental ones,
it would be interesting to add to the model irregularity of the supporting surface
and/or forces and torques of rolling resistance.

\section*{Acknowledgments}
The authors thank Ivan Mamaev for useful comments and discussions.

\section*{Funding}
The work of A. Kilin (Sections~\ref{sec1}--\ref{sec3}) was carried out within the framework of the state assignment of the Ministry of Science and Higher Education of Russia (FEWS-2020-0009).
 The work of E. Pivovarova (Sections~\ref{sec4} and~\ref{sec5}) is supported by the Russian Science Foundation under grant 23-71-01045.

\bibliographystyle{myunsrt}
\bibliography{biblio}

\end{document}